\documentclass{amsart}

\usepackage{amsmath,float,graphicx}

\newtheorem{theorem}{Theorem}[section]
\newtheorem{lemma}[theorem]{Lemma}

\newtheorem{proposition}[theorem]{Proposition}
\theoremstyle{remark}
\newtheorem*{remark}{Remark}

\title{The cap set problem: 41-cap 5-flats}

\author{Henry (Maya) Robert Thackeray}

\address{Department of Mathematics and Applied Mathematics, University of Pretoria, Pretoria, 0002 South Africa, maya.thackeray@up.ac.za, mayart314@outlook.com}

\begin{document}
\begin{abstract}
An \emph{$s$-cap $n$-flat}, or an \emph{$n$-dimensional cap} of \emph{size} $s$, is a pair $(S,F)$ where $F$ is an $n$-dimensional affine space over $\mathbb{Z}/3\mathbb{Z}$ and the size-$s$ subset $S$ of $F$ contains no triple of collinear points. The cap set problem in dimension $n$ asks for the largest $s$ for which an $s$-cap $n$-flat exists. This series of articles investigates the cap set problem in dimensions up to and including $7$. This is the second paper in the series.

By applying and adapting methods from the first paper in the series, we systematically classify all $5$-dimensional caps of size at least $41$.
\end{abstract}
\maketitle

Keywords: cap, cap set, combinatorics, affine space, projective space

MSC2020: 51E20, 05B40, 05D99, 05B25, 51E15

\section{Introduction}

This paper is the second in a series that investigates the cap set problem. For general dimension $n$, the cap set problem asks for the largest positive integer $s$ such that there exists an \emph{$s$-cap $n$-flat} (also called an \emph{$n$-dimensional cap} of \emph{size} $s$), that is, a pair $(S,F)$ such that $F$ is an $n$-dimensional affine space over $\mathbb{Z}/3\mathbb{Z}$ (that is, $F$ is an \emph{$n$-flat}) and the size-$s$ subset $S$ of $F$ contains no triple of collinear points. For an $s$-cap $n$-flat $C$, we write $|C|$ for the size $s$. The cap set problem has been solved for dimensions up to and including $6$ so far. For background and literature results on the cap set problem, see for example Davis and Maclagan \cite{DM03}, Edel et al.\@ \cite{EFLS02}, Ellenberg and Gijswijt \cite{EG17}, and Potechin \cite{P08}.

In the first paper of this series (Thackeray \cite{T21}), certain caps in dimensions up to and including $4$ were classified using what were called standard diagrams. The current paper applies and adapts those techniques to dimension $5$, thus systematically classifying $s$-cap $5$-flats with $s \geq 41$.

We take the paper of Thackeray \cite{T21} as known. In particular, its definitions of a standard diagram and an $\{a,b,c\}$ hyperplane direction for a multiset $\{a,b,c\}$ of nonnegative integers, as well as the names of caps in that paper (square, square pyramid, $990A_{1}$, and so on), are used throughout the current paper. A \emph{complete} cap is a cap $(S,F)$ such that if $(S',F)$ is also a cap and $S \subseteq S'$, then $S = S'$. We use the following definitions that relate to $3$-dimensional caps: the \emph{centre} of a cube minus long diagonal is the unique point that is the midpoint of three different cap line segments, and for a tetrahedron plus centre $C$ with its three $\{2,2,1\}$ hyperplane directions, the \emph{centre} of $C$ is the intersection of the $1$-cap $2$-flats in those three $2$-flat directions.

\section{Lemmas}\label{secrestr}

We prepare for later arguments by proving the following lemmas.

\begin{lemma}\label{n4s17m3}
We have the following.
\begin{itemize}
\item[(a)] If $C$ is a $20$- or $19$-cap $4$-flat, then $C$ has at least one $\{9,9,2\}$ or $\{9,9,1\}$ hyperplane direction.
\item[(b)] If $C$ is a $18$-cap $4$-flat with no $\{9,9,0\}$ hyperplane directions and no $\{9,8,1\}$ hyperplane directions, then $C$ has at least one $\{8,8,2\}$ hyperplane direction $D$ such that at least one of the two $8$-cap $3$-flats in $D$ is a saddled cube.
\item[(c)] If $C$ is a $17$-cap $4$-flat, then $C$ has at least one $\{9,8,0\}$, $\{9,7,1\}$, or $\{8,8,1\}$ hyperplane direction.
\end{itemize}
\end{lemma}
\begin{proof}
Parts (a) and (b) follow immediately from the classification of $s$-cap $4$-flats with $s \geq 18$ in Thackeray \cite{T21}. It remains to prove part (c).

The standard diagram in Figure \ref{fign4s17} implies that every $17$-cap $4$-flat has a $\{9,8,0\}$, $\{9,7,1\}$, $\{8,8,1\}$, $\{8,7,2\}$, or $\{7,7,3\}$ hyperplane direction: if some $17$-cap $4$-flat had no such hyperplane direction, then all points $P_{D}$ in the standard diagram would be on or above the line $L$, but the centre of mass of the points $P_{D}$ is $P_{Cr}$, which is below $L$. (In all standard diagrams, we follow notation from Thackeray \cite{T21}: the point $P_{Cr}$ is denoted $+$, and $d$ is the signed vertical distance above the line $L$ in the $(x,y)$ diagram; for example, we have $d = y - (9/2)x + 130$ in Figure \ref{fign4s17}.)
\begin{figure}
\centering
\includegraphics{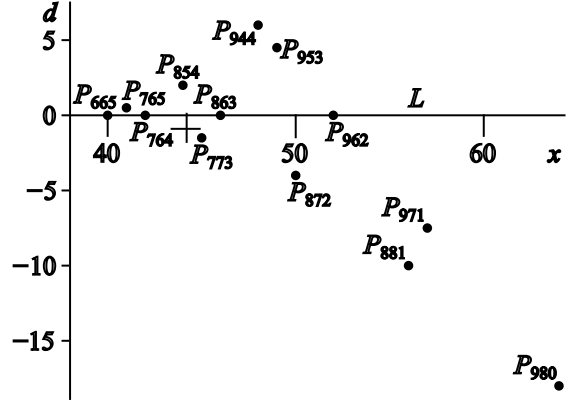}
\caption{Standard diagram for $(n,s) = (4,17)$. The line $L$ has equation $2y = 9x - 260$ and passes through $P_{962}$, $P_{863}$, $P_{764}$, and $P_{665}$.\label{fign4s17}}
\end{figure}

Suppose that some $17$-cap $4$-flat $C$ has no $\{9,8,0\}$, no $\{9,7,1\}$, and no $\{8,8,1\}$ hyperplane directions. It follows that for some co-ordinate $x_{1}$ of $C$, the $3$-flat $x_{1} = -1$ has exactly seven cap points and the $3$-flat $x_{1} = 1$ has exactly seven or exactly eight cap points. Thackeray \cite{T21} gave representatives of the three isomorphism classes of $8$-cap $3$-flats and the two isomorphism classes of $7$-cap $3$-flats. Without loss of generality, the cap in the $3$-flat $x_{1} = -1$ is a representative $7$-cap $3$-flat and the cap in the $3$-flat $x_{1} = 1$ is the image under an invertible linear map $T$ of some representative $7$- or $8$-cap $3$-flat.

For each $T$ and each choice of representative $3$-flats, a computer search found all possible sets of cap points in the $3$-flat $x_{1} = 0$: the search found each point $P$ in that $3$-flat that is not a midpoint of a line segment consisting of a cap point in $x_{1} = -1$ and a cap point in $x_{1} = 1$, and in the case where $x_{1} = 1$ has exactly seven cap points, all lines of three points $P$ were taken into account. In each case, the computer confirmed that:
\begin{itemize}
\item if the $3$-flat $x_{1} = 1$ has exactly seven cap points, then $C$ has a $\{9,8,0\}$, $\{9,7,1\}$, $\{8,8,1\}$, or $\{8,7,2\}$ hyperplane direction; and
\item if the $3$-flat $x_{1} = 1$ has exactly eight cap points, then $C$ has a $\{9,8,0\}$, $\{9,7,1\}$, or $\{8,8,1\}$ hyperplane direction.
\end{itemize}
Therefore, if $C$ has a $\{7,7,3\}$ hyperplane direction, then $C$ has a $\{9,8,0\}$, $\{9,7,1\}$, $\{8,8,1\}$, or $\{8,7,2\}$ hyperplane direction, and if $C$ has an $\{8,7,2\}$ hyperplane direction, then $C$ has a $\{9,8,0\}$, $\{9,7,1\}$, or $\{8,8,1\}$ hyperplane direction. We obtain a contradiction, and the proof is done.
\end{proof}

In the same way, computer searches were used to generate lists of $4$-dimensional caps for which some hyperplane direction $D$ has a prescribed point count and two of the hyperplanes in $D$ contain caps in prescribed isomorphism classes, for example, a $16$-cap $4$-flat with an $\{8,3,5\}$ hyperplane direction in which the $8$- and $5$-cap $3$-flats are a saddled cube and a tetrahedron plus centre respectively. These lists of caps are in the supplementary files. Up to isomorphism, these lists contain all possibilities for certain prescribed hyperplane point counts where two of the hyperplanes contain caps in prescribed isomorphism classes, although some lists may contain duplicates up to isomorphism.

\begin{lemma}\label{n4s18m30}
Let $C$ be a $20$-, $19$-, or $18$-cap $4$-flat with point count $\{a,b,c\}$ for some $3$-flat direction $D$.
\begin{itemize}
\item[(a)] If $\{a,b,c\}$ is among $\{9,9,2\}$, $\{9,9,1\}$, and $\{9,8,2\}$, then each $9$-cap $3$-flat in $D$ is a square antiprism plus centre and each $8$-cap $3$-flat in $D$ is a square antiprism.
\item[(b)] If $\{a,b,c\}$ is among $\{8,6,6\}$, $\{8,6,5\}$, and $\{7,6,6\}$, then each $8$-, $7$-, $6$-, or $5$-cap $3$-flat in $D$ is respectively a cube, a cube minus point, a cube minus long diagonal, or a square pyramid.
\item[(c)] If $\{a,b,c\} = \{9,6,3\}$, then the $9$-cap $3$-flat in $D$ is a square antiprism plus centre and the $6$-cap $3$-flat in $D$ is a cube minus face diagonal (if $C$ is a $963A$) or a cube minus edge (otherwise).
\end{itemize}
\end{lemma}

\begin{proof}
Check the representative caps in Thackeray \cite{T21} directly.
\end{proof}

\begin{lemma}[Point counts]\label{n4m3} Let $\{a,b,c\}$ be the point count of a $4$-dimensional cap $C$ for some $3$-flat direction of $C$ (not necessarily with $a \geq b \geq c$).
\begin{itemize}
\item[(a)] If $(a,b)$ is among $(9,9)$, $(9,8)$, $(9,7)$, and $(8,8)$, then $c \leq 2$.
\item[(b)] If $(a,b)$ is $(9,6)$ or $(8,7)$, then $c \leq 3$.
\item[(c)] If $(a,b)$ is $(9,5)$ or $(7,7)$, then $c \leq 4$.
\item[(d)] If $(a,b) = (9,4)$, then $c \leq 5$.
\end{itemize}
\end{lemma}
\begin{proof}
From Thackeray \cite{T21}, the following results hold: $|C| \leq 20$; if $|C| = 20$, then each $3$-flat direction of $C$ has point count $\{9,9,2\}$ or $\{8,6,6\}$; and if $|C| = 19$, then each $3$-flat direction of $C$ has point count $\{9,9,1\}$, $\{9,8,2\}$, $\{8,6,5\}$, or $\{7,6,6\}$. At least one of those results is violated if (a), (b), (c), or (d) fails.
\end{proof}

\begin{lemma}[No 982/866, 982/865, 963/866, 963/865, and 963/766 clashes]\label{n5clash} No $3$-flat point count of a $5$-dimensional cap is among
\[\begin{array}{c}
\displaystyle \left(\begin{array}{ccc}
9 & 8 & 2\\
* & 6 & *\\
* & 6 & *
\end{array}\right),
\left(\begin{array}{ccc}
9 & 8 & 2\\
* & 6 & *\\
* & 5 & *
\end{array}\right),
\left(\begin{array}{ccc}
* & 8 & *\\
9 & 6 & 3\\
* & 6 & *
\end{array}\right),\\
\displaystyle \left(\begin{array}{ccc}
* & 8 & *\\
9 & 6 & 3\\
* & 5 & *
\end{array}\right)\textrm{, and }
\left(\begin{array}{ccc}
* & 7 & *\\
9 & 6 & 3\\
* & 6 & *
\end{array}\right).
\end{array}\]
\end{lemma}
\begin{proof}
By Lemma \ref{n4s18m30}, each of the first two point counts above yields an $8$-cap $3$-flat that is both a cube and a square antiprism, and each other point count above yields a $6$-cap $3$-flat that both is and is not a cube minus long diagonal.
\end{proof}

\begin{lemma}[No 865/865 clashes]\label{n5865865}
No $3$-flat point count of a $5$-dimensional cap is
\[\left(\begin{array}{ccc}
8 & * & 5\\
1 & 6 & *\\
8 & * & 5
\end{array}\right)
\textrm{ or }
\left(\begin{array}{ccc}
8 & * & 6\\
1 & 5 & *\\
8 & * & 6
\end{array}\right).\]
\end{lemma}
\begin{proof}
If one of the matrices above is the point count of some $5$-dimensional cap for some independent pair $(x_{1},x_{2})$ of co-ordinates, then the structure of the $19$-cap $4$-flats $x_{1} = \pm x_{2}$ and their $3$-flat directions with point count $\{8,6,5\}$ implies that the cubes $(x_{1},x_{2}) = (-1,\pm 1)$ are translations of each other, so the $1$-cap $3$-flat $(x_{1},x_{2}) = (-1,0)$ has no cap points.
\end{proof}

For any nonnegative integer $n$, any $n$-flat $F'$, and any cap $(S,F)$ such that $F'$ contains $F$, a \emph{point reflection} of $(S,F)$ (\emph{in} $F'$) is a cap $(\widetilde{S},\widetilde{F})$ such that for some point $O$ in $F'$, it is true that $\widetilde{S}$ and $\widetilde{F}$ are the images of $S$ and $F$ respectively under the affine transformation of $F'$ defined as follows: $O$ is sent to $O$, and each point $P$ in $F' - \{O\}$ is sent to the point $Q$ in $F' - \{O\}$ such that the midpoint of $PQ$ is $O$.

\begin{lemma}[No triangles of 919 edges and 919/928 triangles] \label{n5no919928}
No $3$-flat point count of a $5$-dimensional cap is of the form
\[\left(\begin{array}{ccc}
9 & 1 & 9\\
1 & 1 & *\\
9 & * & *
\end{array}\right)
\textrm{ or }
\left(\begin{array}{ccc}
9 & 1 & 9\\
2 & 2 & *\\
8 & * & *
\end{array}\right).\]
These forbidden point counts are a \emph{triangle of $919$ edges} and a \emph{$919/928$ triangle} respectively.
\end{lemma}
\begin{proof}
In each $\{9,9,1\}$ hyperplane direction of an arbitrary $19$-cap $4$-flat, each of the two $9$-cap $3$-flats is a point reflection of the other. For each $\{9,8,2\}$ hyperplane direction $D$ of an arbitrary $19$-cap $4$-flat, there is a unique way to enlarge the $8$-cap $3$-flat in $D$ to a $9$-cap $3$-flat by adding a ninth cap point, and that new $9$-cap $3$-flat is a point reflection of the $9$-cap $3$-flat in $D$.

Suppose that the point count of some $5$-dimensional cap for some independent pair $(x_{1},x_{2})$ of co-ordinates is a triangle of $919$ edges. The $9$-cap $3$-flat $(x_{1},x_{2}) = (-1,1)$ is a point reflection of the $9$-cap $3$-flat $(x_{1},x_{2}) = (-1,-1)$, which is a point reflection of $(x_{1},x_{2}) = (1,1)$, which is a point reflection of $(x_{1},x_{2}) = (-1,1)$. Therefore, the $9$-cap $3$-flat $(x_{1},x_{2}) = (-1,1)$ is a point reflection of itself -- but it is not: in each of its $\{3,3,3\}$ hyperplane directions, the three $3$-cap $2$-flats are translations of one another, but no $3$-cap $2$-flat is a point reflection of itself.

For a $919/928$ triangle, replace the $8$-cap $3$-flat $(x_{1},x_{2}) = (-1,-1)$ with its unique completion to a $9$-cap $3$-flat, and apply the same argument.
\end{proof}

\section{Computer searches for some 5-dimensional caps}

We consider some large $5$-dimensional caps. In particular, we consider a $45$-cap $5$-flat, a $42$-cap $5$-flat that we call the representative $\Delta{}686$, and five $41$-cap $5$-flats that we call the representative $41A$, the representative $41B$, the representative $41C$, the representative $41D$, and the representative $41E$. A $\Delta{}686$ (respectively a $41A$, a $41B$, \ldots, a $41E$) is defined to be an arbitrary cap that is isomorphic to the representative $\Delta{}686$ (respectively the representative $41A$, and so on).

\begin{lemma}[Large 5D caps]\label{n5m4largesmall}
Let $C$ be a $5$-dimensional cap.
\begin{itemize}
\item[(a)] Let the co-ordinate $x_{1}$ of $C$ be such that each $4$-flat $x_{1} = a$ has exactly $c_{a}$ cap points, for $a$ in $\{-1,0,1\}$. Suppose $c_{-1} \geq 19$ and $c_{1} \geq 18$. It follows that $|C| \leq 41$, and if also $(c_{-1},c_{1}) \neq (20,20)$, then $|C| \leq 40$.
\end{itemize}
For parts (b) to (f), let the co-ordinate $x_{1}$ of $C$ be such that the $4$-flats $x_{1} = \pm 1$ have exactly $18$ cap points each.
\begin{itemize}
\item[(b)] We have the following. (i) If the $18$-cap $4$-flats $x_{1} = \pm 1$ are both $882A_{1}$ caps or both $882A_{2}$ caps, then $|C| \leq 45$. (ii) If those $18$-cap $4$-flats are both $963B$ caps, then $|C| \leq 42$. (iii) If those $18$-cap $4$-flats are two $990A_{1}$ caps, or a $990A_{2}$ and an $882A_{2}$, or two $981A$ caps, or two $972A$ caps, or a $954A$ and an $882A_{1}$, or a $981I$ and an $882A_{2}$, then $|C| \leq 41$. (iv) If they are in none of cases (i) to (iii), then $|C| \leq 40$.
\item[(c)] Let the $18$-cap $4$-flats $x_{1} = \pm 1$ be two $882A_{1}$ caps. Suppose $|C| \geq 41$. There are co-ordinates $x_{2}$ to $x_{5}$, with $(x_{1},\ldots,x_{5})$ independent, such that, possibly after negating $x_{1}$, all cap points of $C$ are among the cap points in Figure \ref{fign5s45882A1}. Also, each of the nine $\{8,8,2\}$ hyperplane directions of each $4$-flat $x_{1} = \pm 1$ is parallel to one of the nine $\{8,5,5\}$ hyperplane directions of the other $4$-flat $x_{1} = \mp 1$.

\begin{figure}[tbph]
\centering
\includegraphics{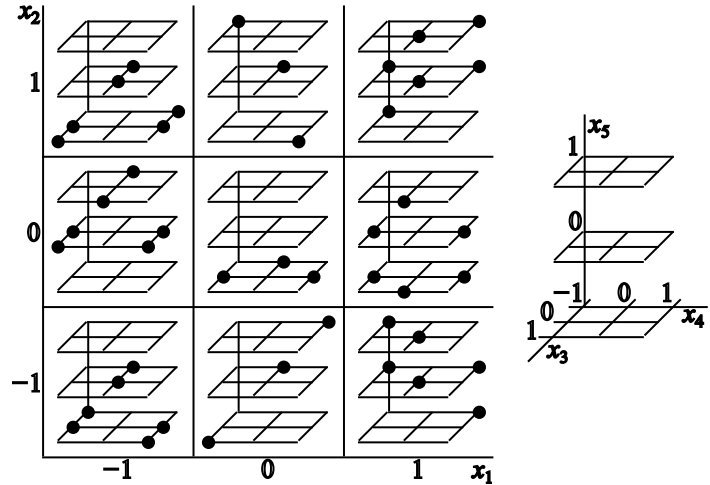}
\caption{A $45$-cap $5$-flat.\label{fign5s45882A1}}
\end{figure}

\item[(d)] Let the $18$-cap $4$-flats $x_{1} = \pm 1$ be two $882A_{2}$ caps. Suppose $|C| \geq 42$. There are co-ordinates $x_{2}$ to $x_{5}$, with $(x_{1},\ldots,x_{5})$ independent, such that all cap points of $C$ are among the cap points in Figure \ref{fign5s45m0}. That figure is the image of Figure \ref{fign5s45882A1} under the transformation
\[(x_{1},\ldots,x_{5}) \mapsto (1 - x_{3} - x_{5},x_{5} - x_{3},x_{1} + x_{3} - x_{5},x_{2},-x_{4}).\]

\begin{figure}[tbph]
\centering
\includegraphics{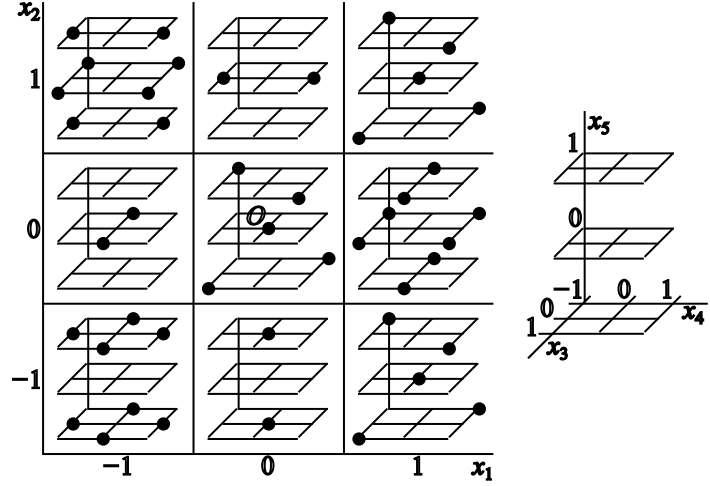}
\caption{The representative $45$-cap $5$-flat.\label{fign5s45m0}}
\end{figure}

\item[(e)] Let the $18$-cap $4$-flats $x_{1} = \pm 1$ be two $882A_{2}$ caps. Suppose $41 \leq |C| \leq 44$. Suppose that $C$ is complete. It follows that $|C| = 41$ and there are co-ordinates $x_{2}$ to $x_{5}$, with $(x_{1},\ldots,x_{5})$ independent, such that $C$ is the representative $41A$ in Figure \ref{fign5s41A}.

\begin{figure}
\centering
\includegraphics{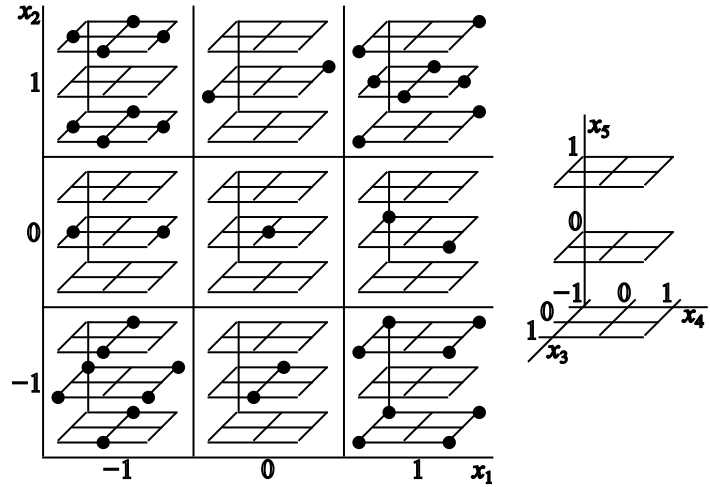}
\caption{The representative $41A$.\label{fign5s41A}}
\end{figure}

\item[(f)] Suppose $|C| \geq 41$. Let the $18$-cap $4$-flats $x_{1} = -1$ and $x_{1} = 1$ be (i) two $963B$ caps, (ii) two $981A$ caps, (iii) a $990A_{2}$ and an $882A_{2}$ respectively, (iv) a $954A$ and an $882A_{1}$ respectively, or (v) a $981I$ and an $882A_{2}$ respectively. It follows that there are co-ordinates $x_{2}$ to $x_{5}$, with $(x_{1},\ldots,x_{5})$ independent, such that $C$ is respectively (i) the representative $\Delta{}686$ in Figure \ref{fign5s42Delta686} or that cap minus one cap point, (ii) the representative $41B$ in Figure \ref{fign5s41B} or its image under $(x_{1},\ldots,x_{5}) \mapsto (-x_{1},x_{2},x_{3},x_{4},x_{5})$, (iii) the representative $41C$ in Figure \ref{fign5s41C}, (iv) the representative $41D$ in Figure \ref{fign5s41D}, or (v) the representative $41E$ in Figure \ref{fign5s41E}.

\begin{figure}
\centering
\includegraphics{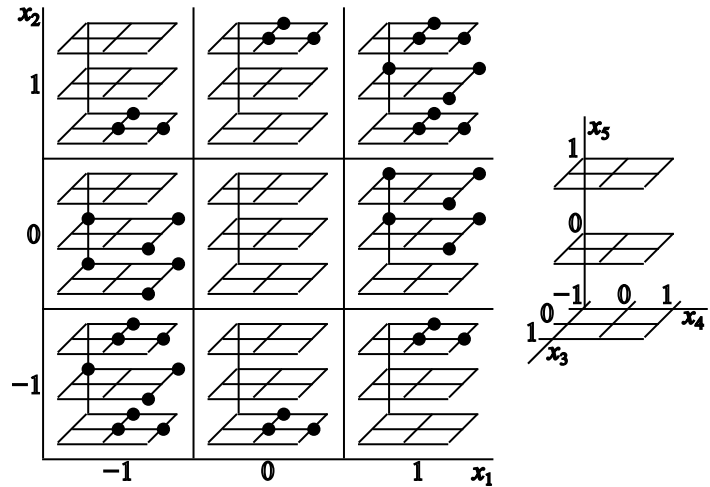}
\caption{The representative $\Delta{}686$.\label{fign5s42Delta686}}
\end{figure}

\begin{figure}
\centering
\includegraphics{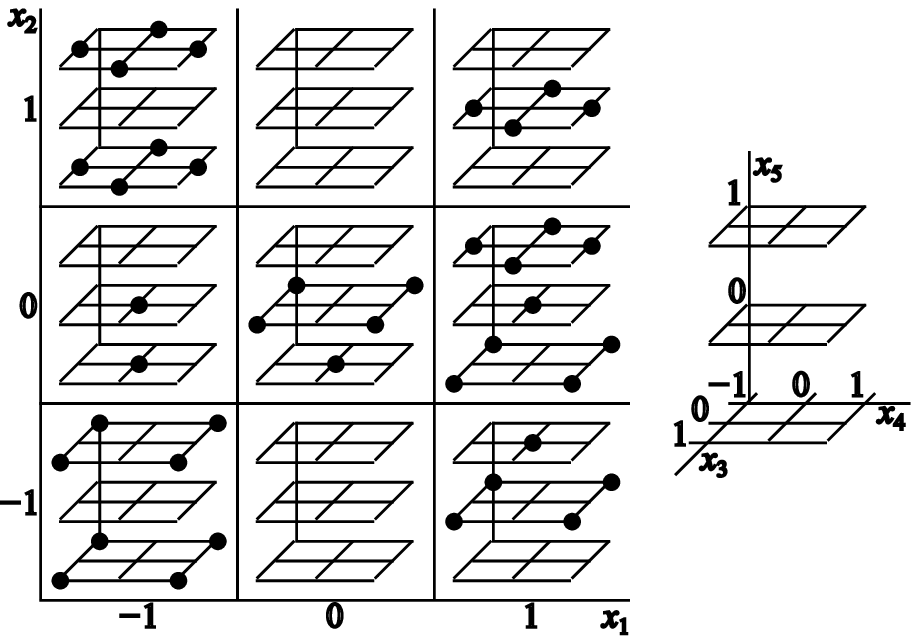}
\caption{The representative $41B$.\label{fign5s41B}}
\end{figure}

\begin{figure}
\centering
\includegraphics{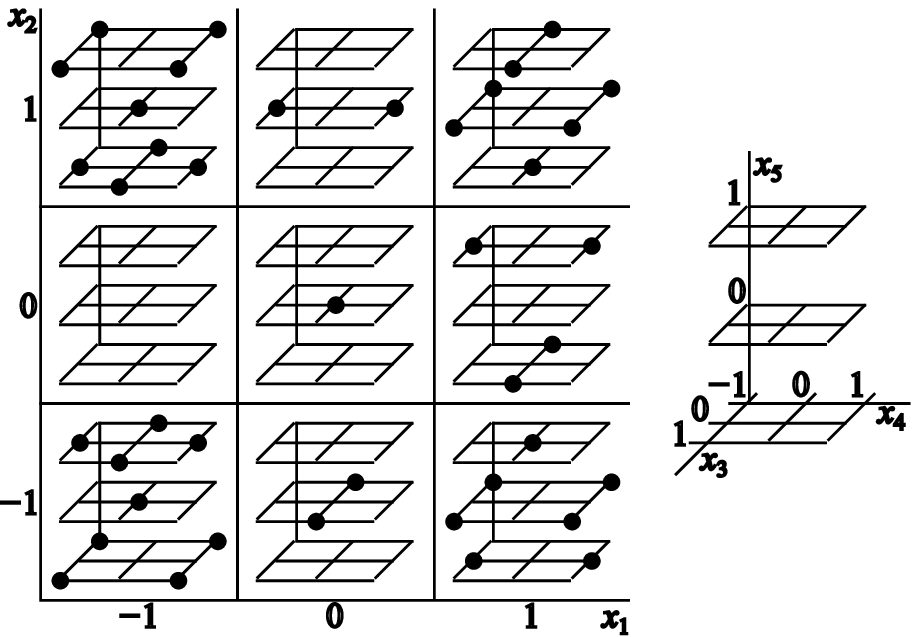}
\caption{The representative $41C$.\label{fign5s41C}}
\end{figure}

\begin{figure}
\centering
\includegraphics{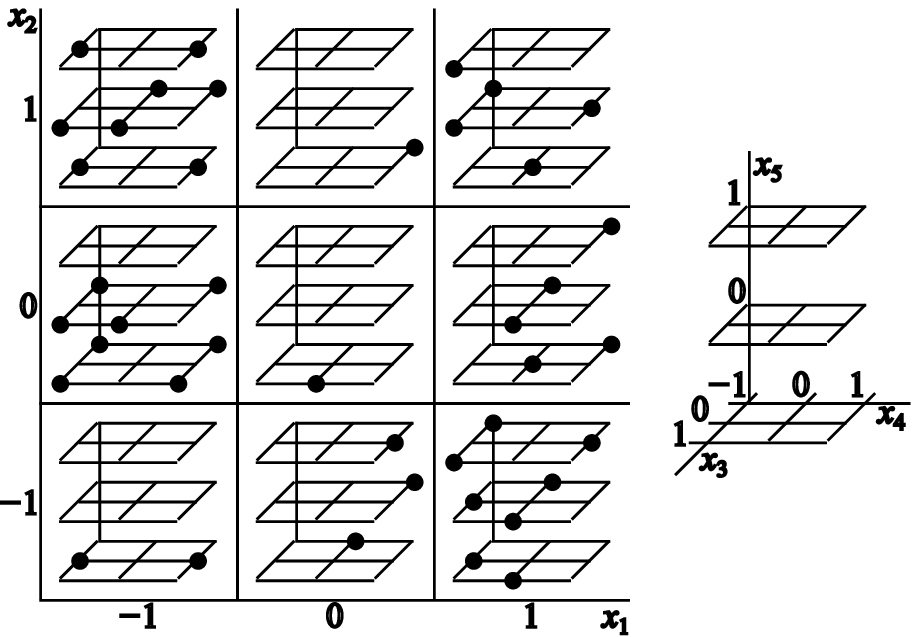}
\caption{The representative $41D$.\label{fign5s41D}}
\end{figure}

\begin{figure}
\centering
\includegraphics{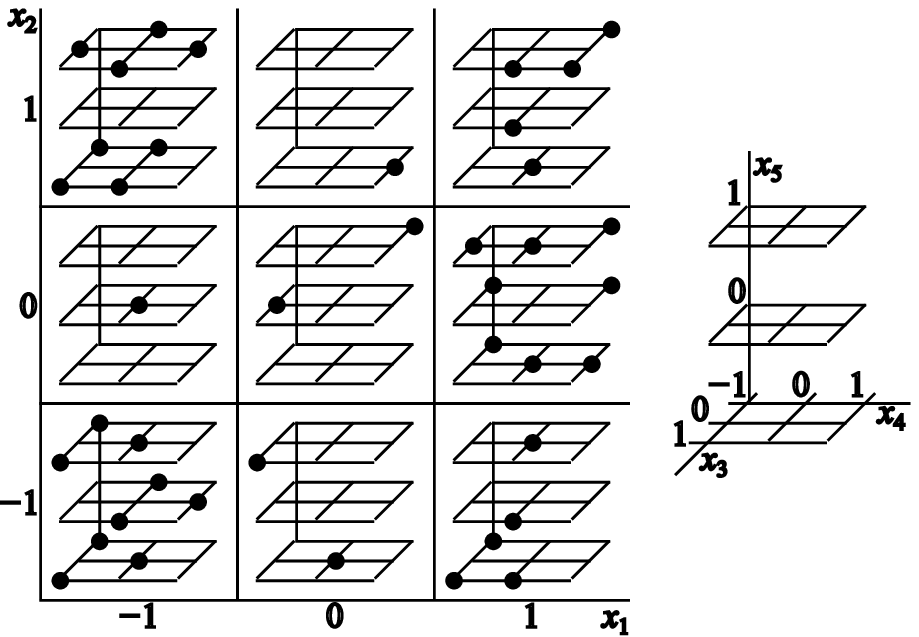}
\caption{The representative $41E$.\label{fign5s41E}}
\end{figure}

\item[(g)] Suppose $|C| \geq 41$. Let the co-ordinate $x_{2}$ of $C$ be such that the $4$-flats $x_{2} = \pm 1$ have exactly $18$ cap points each, and the $18$-cap $4$-flats $x_{2} = \pm 1$ are (i) two $990A_{1}$ caps or (ii) two $972A$ caps. It follows that $|C| = 41$ and there are co-ordinates $x_{1}$, $x_{3}$, $x_{4}$, and $x_{5}$ such that $C$ is respectively (i) the $41A$ in Figure \ref{fign5s41A} or (ii) the $41C$ in Figure \ref{fign5s41C}.
\end{itemize}
\end{lemma}

\begin{proof}
(a) It is enough to show that $(c_{-1},c_{0},c_{1})$ is not among $(20,2,19)$, $(20,3,18)$, $(19,3,19)$, and $(19,4,18)$.

For each of these four $(c_{-1},c_{0},c_{1})$ values, the cap in the $4$-flat $x_{1} = -1$ is without loss of generality a fixed $20$- or $19$-cap $4$-flat (see Thackeray \cite{T21}), and the cap in the $4$-flat $x_{1} = 1$ is without loss of generality the image under an invertible linear map $T$ of some representative $19$- or $18$-cap $4$-flat in Thackeray \cite{T21}. For each $T$ and each choice of representative $c_{-1}$- and $c_{1}$-cap $4$-flats, a computer search found the maximum number of cap points in the $4$-flat $x_{1} = 0$: the search found each point $P$ that was not a midpoint of a line segment consisting of a cap point in $x_{1} = -1$ and a cap point in $x_{1} = 1$, and then took into account any lines of three points $P$. In each case, the maximum number of cap points in the $4$-flat $x_{1} = 0$ was at most $c_{0} - 1$. See Table \ref{tblcompsearch}.

(b) A second computer search similarly chose an $18$-cap $4$-flat from each of the $20$ isomorphism classes of $18$-cap $4$-flats as a representative, and looked at the $\binom{21}{2}$ pairs of isomorphism classes of $18$-cap $4$-flats for $x_{1} = \pm 1$ (including cases where the $18$-cap $4$-flats $x_{1} = \pm 1$ are isomorphic to each other), applying an invertible linear map $T$ to the $4$-flat $x_{1} = 1$. See Table \ref{tblcompsearch}.

\setlength{\tabcolsep}{2pt}
\begin{table}
\begin{tabular}{ccccccccccccccccccccccc}
\hline
& \rotatebox{90}{$990A_{1}$} & \rotatebox{90}{$990A_{2}$} & \rotatebox{90}{$990A_{3}$} & \rotatebox{90}{$990B$} & \rotatebox{90}{$981A$} & \rotatebox{90}{$981B$} & \rotatebox{90}{$981C$} & \rotatebox{90}{$981D$} & \rotatebox{90}{$981E$} & \rotatebox{90}{$981F$} & \rotatebox{90}{$981G$} & \rotatebox{90}{$981H$} & \rotatebox{90}{$981I$} & \rotatebox{90}{$981J$} & \rotatebox{90}{$972A$} & \rotatebox{90}{$963A$} & \rotatebox{90}{$963B$} & \rotatebox{90}{$954A$} & \rotatebox{90}{$882A_{1}$} & \rotatebox{90}{$882A_{2}$} & \rotatebox{90}{$19$-cap\mbox{ }} & \rotatebox{90}{$20$-cap\mbox{ }}\\
\hline
$20$-cap & $2$ & $2$ & $1$ & $2$ & $2$ & $2$ & $2$ & $2$ & $2$ & $2$ & $2$ & $2$ & $2$ & $2$ & $2$ & $2$ & $1$ & $2$ & $2$ & $2$ & $1$ & $1$\\
$19$-cap & $2$ & $3$ & $2$ & $3$ & $3$ & $3$ & $3$ & $3$ & $3$ & $3$ & $3$ & $3$ & $3$ & $3$ & $3$ & $3$ & $3$ & $3$ & $2$ & $3$ & $2$ & \\
$882A_{2}$ & $4$ & $5$ & $4$ & $4$ & $3$ & $4$ & $4$ & $4$ & $4$ & $4$ & $4$ & $4$ & $5$ & $3$ & $4$ & $4$ & $3$ & $4$ & $3$ & $9$ & & \\
$882A_{1}$ & $4$ & $3$ & $4$ & $3$ & $4$ & $3$ & $3$ & $3$ & $3$ & $3$ & $3$ & $3$ & $3$ & $3$ & $4$ & $4$ & $4$ & $5$ & $9$ & & & \\
$954A$ & $4$ & $4$ & $4$ & $4$ & $4$ & $4$ & $3$ & $3$ & $4$ & $3$ & $4$ & $4$ & $4$ & $4$ & $4$ & $4$ & $4$ & $4$ & & & & \\
$963B$ & $4$ & $3$ & $4$ & $4$ & $4$ & $3$ & $4$ & $4$ & $4$ & $4$ & $4$ & $4$ & $4$ & $4$ & $4$ & $4$ & $6$ & & & & & \\
$963A$ & $4$ & $3$ & $4$ & $4$ & $4$ & $4$ & $4$ & $4$ & $4$ & $4$ & $4$ & $4$ & $4$ & $4$ & $4$ & $4$ & & & & & & \\
$972A$ & $4$ & $3$ & $4$ & $4$ & $4$ & $4$ & $4$ & $4$ & $4$ & $4$ & $4$ & $4$ & $4$ & $4$ & $5$ & & & & & & & \\
$981J$ & $3$ & $3$ & $3$ & $3$ & $4$ & $3$ & $3$ & $4$ & $4$ & $4$ & $4$ & $4$ & $4$ & $4$ & & & & & & & & \\
$981I$ & $3$ & $4$ & $3$ & $4$ & $4$ & $4$ & $4$ & $4$ & $4$ & $4$ & $4$ & $4$ & $4$ & & & & & & & & & \\
$981H$ & $4$ & $4$ & $3$ & $4$ & $4$ & $4$ & $4$ & $4$ & $4$ & $4$ & $4$ & $4$ & & & & & & & & & & \\
$981G$ & $4$ & $3$ & $3$ & $3$ & $4$ & $3$ & $4$ & $4$ & $4$ & $4$ & $4$ & & & & & & & & & & & \\
$981F$ & $4$ & $4$ & $3$ & $4$ & $4$ & $4$ & $4$ & $4$ & $4$ & $4$ & & & & & & & & & & & & \\
$981E$ & $4$ & $3$ & $3$ & $4$ & $4$ & $4$ & $4$ & $4$ & $4$ & & & & & & & & & & & & & \\
$981D$ & $4$ & $3$ & $4$ & $4$ & $4$ & $4$ & $4$ & $4$ & & & & & & & & & & & & & & \\
$981C$ & $4$ & $3$ & $3$ & $4$ & $4$ & $4$ & $4$ & & & & & & & & & & & & & & & \\
$981B$ & $4$ & $4$ & $4$ & $4$ & $4$ & $4$ & & & & & & & & & & & & & & & & \\
$981A$ & $4$ & $3$ & $3$ & $4$ & $5$ & & & & & & & & & & & & & & & & & \\
$990B$ & $4$ & $3$ & $3$ & $4$ & & & & & & & & & & & & & & & & & & \\
$990A_{3}$ & $3$ & $3$ & $4$ & & & & & & & & & & & & & & & & & & & \\
$990A_{2}$ & $2$ & $4$ & & & & & & & & & & & & & & & & & & & & \\
$990A_{1}$ & $5$ & & & & & & & & & & & & & & & & & & & & & \\
\hline
\end{tabular}
\caption{Maximum sizes of subcaps in the $4$-flat $x_{1} = 0$ for pairs of isomorphism classes of $4$-dimensional caps in $x_{1} = \pm 1$, obtained by computer search.\label{tblcompsearch}}
\end{table}

In parts (a) and (b), for each case in the computer search, if at least seven points $P$ exist, then exactly nine points $P$ exist and they form a cap.

(c) Instead of using, say, the representative $882A_{1}$ from Thackeray \cite{T21} twice in the computer search of part (b), use the $882A_{1}$ caps in the $4$-flats $x_{1} = \pm 1$ in Figure \ref{fign5s45882A1}. The computer search verified that there are $144$ maps $T$ for which the $4$-flat $x_{1} = 0$ can have at least five cap points, and for each of those maps $T$, there are exactly nine points $P$ and those points $P$ form a cap. Those maps $T$ form two equivalence classes of equal sizes, where $T_{1}$ and $T_{2}$ are equivalent if and only if $T_{1}^{-1}T_{2}$ is a symmetry of the $18$-cap $4$-flat $x_{1} = 1$ in the figure composed with a translation. The two equivalence classes give, respectively, Figure \ref{fign5s45882A1} and its image under $(x_{1},\ldots,x_{5}) \mapsto (-x_{1},x_{2},-x_{5},x_{4},x_{3})$. The remaining statements of part (c) in the Lemma were verified.

(d) Proceed similarly with the $882A_{2}$ caps in the $4$-flats $x_{1} = \pm 1$ in Figure \ref{fign5s45m0}. The computer search verified that there are $32$ maps $T$ for which the $4$-flat $x_{1} = 0$ can have at least six cap points, and for each of those maps $T$, there are exactly nine points $P$ and they form a cap. Those $T$ form two equivalence classes of equal sizes as in part (c). The two classes give, respectively, Figure \ref{fign5s45m0} and its image under a certain invertible affine transformation (namely, $(x_{1},\ldots,x_{5}) \mapsto (x_{1},x_{2},-x_{3},x_{4},x_{5})$) that restricts to a symmetry of the cap in the $4$-flat $x_{1} = -1$. The isomorphism from Figure \ref{fign5s45882A1} to Figure \ref{fign5s45m0} is easily checked.

(e)-(g) The computer search yielded the required results as before.
\end{proof}

\begin{remark}
It was verified by computer that in Figure \ref{fign5s45m0}, for each point $P$ that is not a cap point, there are exactly five cap line segments of which $P$ is the midpoint. Therefore, if at most four cap points are deleted from the figure to obtain a new cap $C_{1}$ that is also a subcap of some $5$-dimensional cap $C_{2}$, then $C_{2}$ is a subcap of Figure \ref{fign5s45m0}.
\end{remark}
\begin{remark}
The name ``$\Delta{}686$'' comes from the following fact: the structure of $\{8,6,6\}$ hyperplane directions of $20$-cap $4$-flats implies that for a $5$-dimensional cap $C$ with an independent pair $(x_{3},x_{4})$ of co-ordinates, if the point count of $C$ for $(x_{3},x_{4})$ is
\[\left(\begin{array}{ccc}6 & 8 & 6\\ 8 & 8 & *\\ 6 & * & *\end{array}\right),\]
then for some co-ordinates $x_{1}$, $x_{2}$, and $x_{5}$ of $C$, the tuple $(x_{1},\ldots,x_{5})$ is independent and $C$ is the representative $\Delta{}686$ in Figure \ref{fign5s42Delta686v2}. (The $6$-cap $3$-flats $(x_{3},x_{4}) = (-1,\pm 1)$ and $(x_{3},x_{4}) = (1,1)$ are cubes minus long diagonal; without loss of generality, their centres are at $(x_{1},x_{2},x_{5}) = (0,0,0)$ and the $6$-cap $3$-flat $(x_{3},x_{4}) = (-1,-1)$ is as in the figure. All of $C$ is now forced.)

It was verified by computer that in Figure \ref{fign5s42Delta686v2}, for each point $P$ that is not a cap point, there are at least three cap line segments of which $P$ is the midpoint. Therefore, if at most two cap points are deleted from the figure to obtain a new cap $C_{1}$ that is also a subcap of some $5$-dimensional cap $C_{2}$, then $C_{2}$ is a subcap of Figure \ref{fign5s42Delta686v2}.
\end{remark}

\begin{figure}
\centering
\includegraphics{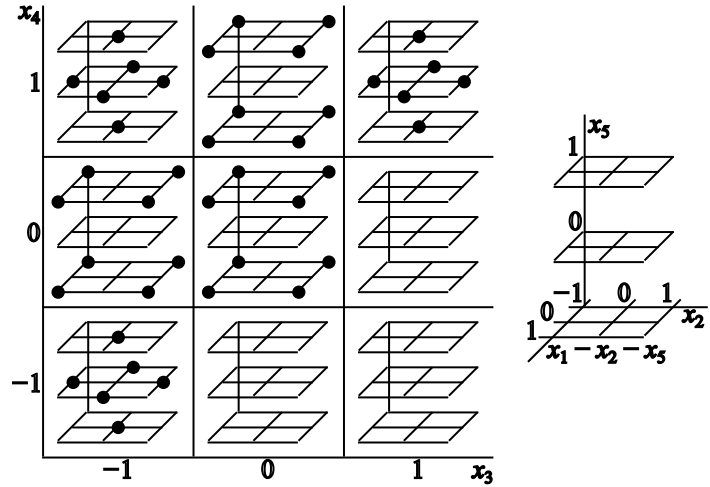}
\caption{The representative $\Delta{}686$, from another viewpoint.\label{fign5s42Delta686v2}}
\end{figure}

\begin{remark}
It was verified by computer that in each of the five representative caps $41A$ to $41E$, each $4$-flat contains at most $18$ cap points.
\end{remark}

\section{Caps of sizes 46 and 45}

Previous literature (see Edel et al.\@ \cite{EFLS02} and Potechin \cite{P08}) has proved the nonexistence of $46$-cap $5$-flats, as well as the uniqueness of $45$-cap $5$-flats up to isomorphism. However, we now briefly prove those results in new ways that illuminate some of the underlying structure involved; these proofs illustrate the type of argument that we use to classify $5$-dimensional caps of size at least $41$.

\begin{theorem}[No 46-cap 5-flats] \label{n5s46m2}
There is no $46$-cap $5$-flat $C$.
\end{theorem}
\begin{proof} By Lemma \ref{n5m4largesmall}, no $4$-flat direction of $C$ has point count among $\{20,20,6\}$, $\{20,19,7\}$, $\{19,19,8\}$, and $\{20,18,8\}$. Therefore, in the standard diagram in Figure \ref{fign5s46}, all points $P_{D}$ are on or above the line $L$, but the centre of mass of the points $P_{D}$ is $P_{Cr}$, which is below $L$.
\end{proof}
\begin{figure}[tbph]
\centering
\includegraphics{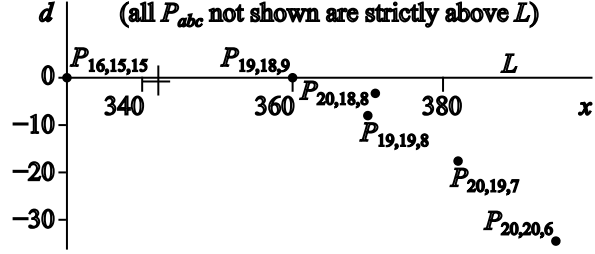}
\caption{Standard diagram for $(n,s) = (5,46)$. The line $L$ has equation $10y = 133x - 29190$ and passes through $P_{19,18,9}$ and $P_{16,15,15}$.\label{fign5s46}}
\end{figure}

We prove the uniqueness of the $45$-cap $5$-flat up to isomorphism. First, we use the standard diagram to determine the point counts of the $4$-flat directions.

\begin{proposition} \label{n5s45m4} In every $45$-cap $5$-flat $C$, the numbers of hyperplane directions with point counts $\{15,15,15\}$ and $\{18,18,9\}$ are $66$ and $55$ respectively.
\end{proposition}
\begin{proof} By Lemma \ref{n5m4largesmall}, no $4$-flat direction of $C$ has point count among $\{20,20,5\}$, $\{20,19,6\}$, $\{19,19,7\}$, $\{20,18,7\}$, and $\{19,18,8\}$. Apply the standard diagram in Figure \ref{fign5s45}.
\end{proof}

\begin{figure}[tbph]
\centering
\includegraphics{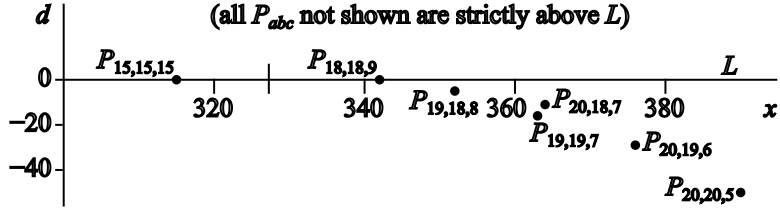}
\caption{Standard diagram for $(n,s) = (5,45)$. The line $L$ has equation $y = 13x - 2730$ and passes through $P_{Cr}$, $P_{18,18,9}$, and $P_{15,15,15}$.\label{fign5s45}}
\end{figure}

The uniqueness of the $45$-cap $5$-flat up to isomorphism quickly follows, by Lemma \ref{n5m4largesmall}.

\begin{theorem}[The 45-cap 5-flat] \label{n5s45m0}
Every $45$-cap $5$-flat is isomorphic to the representative cap in Figure \ref{fign5s45m0}.
\end{theorem}
\begin{proof}
By Proposition \ref{n5s45m4}, every $45$-cap $5$-flat has an $\{18,18,9\}$ hyperplane direction. By Lemma \ref{n5m4largesmall}, every $45$-cap $5$-flat with an $\{18,18,9\}$ hyperplane direction is isomorphic to Figure \ref{fign5s45882A1} or Figure \ref{fign5s45m0}. By the same lemma, the caps in those two figures are isomorphic to each other.
\end{proof}

\begin{remark}
The symmetry group of Figure \ref{fign5s45m0} has size $720$ and acts transitively on the set of cap points. A proof outline follows. Check transitivity directly by using the symmetries
\[\begin{array}{l}
T \colon (x_{1},\ldots,x_{5}) \mapsto (x_{1} - x_{2} + 1,x_{2} + 1,x_{5},x_{3},x_{4}),\\
R \colon (x_{1},\ldots,x_{5}) \mapsto (-x_{1},x_{3},-x_{2},x_{4} + x_{5},-x_{4} + x_{5})\textrm{, and}\\
M = TR^{2}T \colon (x_{1},\ldots,x_{5}) \mapsto (x_{1},-x_{2},x_{3},-x_{5},-x_{4}).
\end{array}\]
It now suffices to show that the stabilizer $G$ of the point $O$ in Figure \ref{fign5s45m0} is the $16$-element group generated by $R$ and $M$. Consider the $3$-flat directions in which the $3$-dimensional caps are eight square pyramids and one tetrahedron plus centre. The only such $3$-flat direction in which $O$ is the centre of the tetrahedron plus centre consists of the $3$-flats $(x_{4},x_{5}) = (a_{4},a_{5})$ for $a_{4}$ and $a_{5}$ in $\mathbb{Z}/3\mathbb{Z}$, so each symmetry in $G$ induces an action on $(x_{4},x_{5})$ space (see the depiction of the $45$-cap $5$-flat in Figure \ref{fign5s45sqpyr}). Now find all possible such actions, and obtain the corresponding symmetries in $G$.
\end{remark}
\begin{figure}[tbph]
\centering
\includegraphics{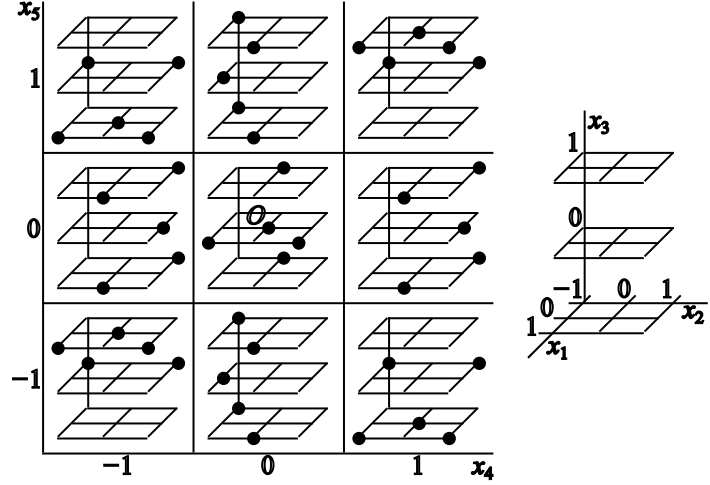}
\caption{The representative $45$-cap $5$-flat, from another viewpoint.\label{fign5s45sqpyr}}
\end{figure}

\section{Remarks on the proof technique for caps of size 41}\label{secproofs41}

The classification of $41$-cap $5$-flats $C$ proceeds according to the following strategy.

We start by using the standard diagram to show that the point count of at least one $4$-flat direction, with an appropriate co-ordinate $x_{1}$, is in some finite set $S_{4}$. The set $S_{4}$ includes the point counts $\{20,20,1\}$ and $\{18,18,5\}$; for the latter, Lemma \ref{n5m4largesmall} implies that $C$ can be obtained from one of Figures \ref{fign5s45m0} through \ref{fign5s41E} by removing at most four cap points.

For each $4$-flat point count in $S_{4}$ that is neither $\{20,20,1\}$ nor $\{18,18,5\}$, we obtain a finite set $S_{3}$ of $3$-flat point counts such that for some co-ordinate $x_{2}$, the point count of the cap for $(x_{1},x_{2})$ is in $S_{3}$. To do this, we exploit $3$-flat point counts of the cap in the $4$-flat $x_{1} = -1$, and we apply Lemmas \ref{n4m3}, \ref{n5clash}, \ref{n5865865}, \ref{n5no919928}, and \ref{n5m4largesmall}; in particular, we apply Lemma \ref{n5m4largesmall} by excluding each $(x_{1},x_{2})$ point count that yields an $\{a,b,c\}$ hyperplane direction of $C$ with $a \geq 19$ and $b \geq 18$.

Finally, we show that for each $4$-flat point count in $S_{4}$ other than $\{20,20,1\}$ and $\{18,18,5\}$, each $3$-flat point count in the corresponding set $S_{3}$
\begin{itemize}
\item is impossible,
\item yields a $4$-flat direction with point count $\{20,20,1\}$ or $\{18,18,5\}$, or
\item implies that either $C$ is isomorphic to one of four complete representative caps that will be described later, or $C$ is complete and has a particular form.
\end{itemize}
On several occasions, the following computer search technique is used to find all possible caps up to isomorphism for a specific $3$-flat point count.

Suppose that for some independent pair $(x_{1},x_{2})$ of co-ordinates of $C$, the cap $C$ has the $3$-flat point count
\[\left(\begin{array}{ccc}
n_{-1,1} & n_{0,1} & n_{1,1}\\
n_{-1,0} & n_{0,0} & n_{1,0}\\
n_{-1,-1} & n_{0,-1} & n_{1,-1}
\end{array}\right).\]

First, we find a representative of each possible cap in the $4$-flat $x_{1} = -1$ with a specified restriction of the co-ordinate $x_{2}$ up to isomorphism, as well as a representative of each possible cap in the $4$-flat $x_{2} = -1$ with a specified restriction of the co-ordinate $x_{1}$ up to isomorphism. To do this, we use the classification of $18$-, $19$-, and $20$-cap $4$-flats in Thackeray \cite{T21} and the lists of caps generated in section \ref{secrestr}.

For instance, if the left column of the matrix is $(9\mbox{ }1\mbox{ }9)^{T}$, then the cap in the $4$-flat $x_{1} = -1$ is a $19$-cap $4$-flat in which the $x_{2}$-hyperplane direction has point count $\{9,1,9\}$. If the bottom row of the matrix is $(9 \mbox{ } 7 \mbox{ } 2)$, then there are four options for the cap in the $4$-flat $x_{2} = -1$ up to isomorphism, namely a $981B$, a $981D$, a $981I$, and a $972A$: these are the $18$-cap $4$-flats that have at least one $\{9,7,2\}$ hyperplane direction, and for each of those four caps, the symmetry group acts transitively on the set of $\{9,7,2\}$ hyperplane directions of the cap. In general, we need to consider the possibility that in a $4$-dimensional cap, for some set of two or more $3$-flat directions of that cap with the same point count, no symmetry of the cap sends any $3$-flat direction in that set to any other; in that case, a representative for each $3$-flat direction in that set must be considered.

Next, we compare the representatives for $x_{1} = -1$ with the representatives for $x_{2} = -1$, and we determine which pairs of representatives are compatible at the $3$-flat $(x_{1},x_{2}) = (-1,-1)$, that is, which pairs of representatives yield the same isomorphism class for the cap in that $3$-flat. The representatives can be chosen so that they give the same cap in that $3$-flat (not just the same cap up to isomorphism).

For each pair of compatible representatives, and for each symmetry $T$ of the $3$-flat $(x_{1},x_{2}) = (-1,-1)$, we combine the representative for $x_{2} = -1$ with an image of the representative for $x_{1} = -1$ under an affine transformation of which the restriction to $(x_{1},x_{2}) = (-1,-1)$ is $T$. (For instance, if $n_{-1,-1} = 9$, then there are $144$ symmetries $T$ to check.) We now have a list of options for the left column and bottom row of the matrix combined.

Finally, for each of those options, by considering which positions are not midpoints of known cap line segments, we find the possible caps in the $3$-flat $(x_{1},x_{2}) = (0,0)$, then in the $3$-flat $(x_{1},x_{2}) = (1,0)$, then in the $3$-flat $(x_{1},x_{2}) = (0,1)$, then in the $3$-flat $(x_{1},x_{2}) = (1,1)$.

\section{Caps of size 41}

Throughout this subsection, $C$ is an arbitrary $41$-cap $5$-flat.

\begin{proposition}[4-flats in 41-cap 5-flats] \label{n5s41}
No hyperplane point count of $C$ is among $\{20,19,2\}$, $\{20,18,3\}$, $\{19,19,3\}$, and $\{19,18,4\}$. In $C$, at least one hyperplane direction has point count $\{a,b,c\}$ such that $(a,b,c) = (20,20,1)$, $(a,b,c) = (18,18,5)$, or $20 \geq a \geq 17 \geq b \geq 16 \geq c$.
\end{proposition}
\begin{proof}
Lemma \ref{n5m4largesmall} implies the first sentence; the standard diagram in Figure \ref{fign5s41} yields the second sentence.
\end{proof}
\begin{figure}[tbph]
\centering
\includegraphics{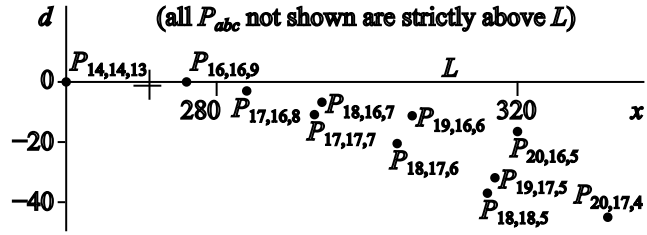}
\caption{Standard diagram for $(n,s) = (5,41)$ after using Lemma \ref{n5m4largesmall} and assuming that the cap has no $\{20,20,1\}$ hyperplane directions. The line $L$ has equation $8y = 95x - 16588$ and passes through $P_{16,16,9}$ and $P_{14,14,13}$.\label{fign5s41}}
\end{figure}

A $41F$ (respectively a $41G$, a $41H$, a $41I$) is defined to be an arbitrary cap that is isomorphic to the corresponding representative among Figures \ref{fign5s41F}, \ref{fign5s41G}, \ref{fign5s41H}, and \ref{fign5s41I}.
\begin{figure}
\centering
\includegraphics{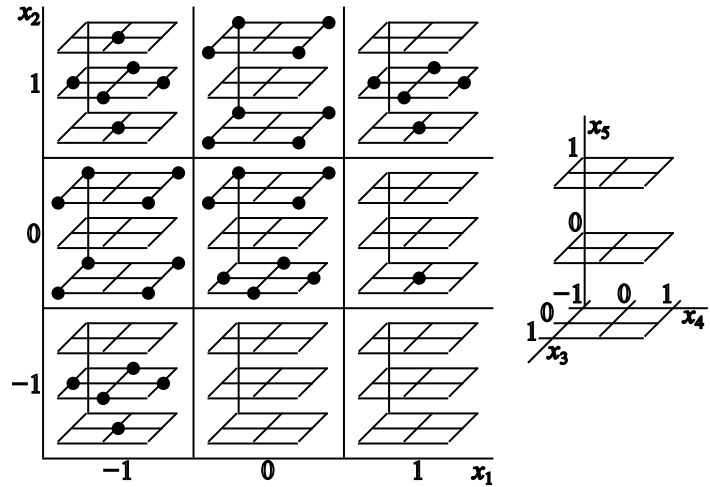}
\caption{The representative $41F$.\label{fign5s41F}}
\end{figure}
\begin{figure}
\centering
\includegraphics{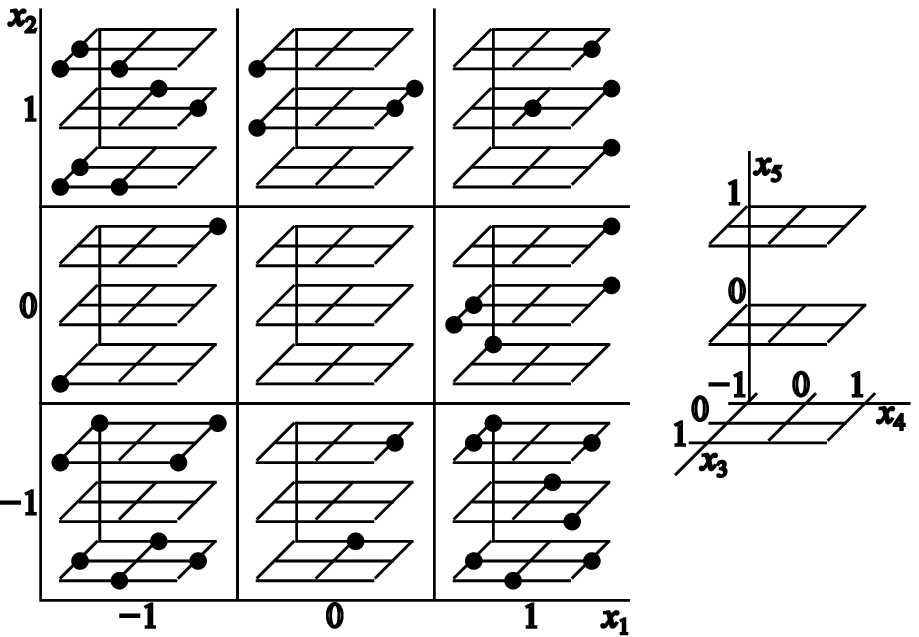}
\caption{The representative $41G$.\label{fign5s41G}}
\end{figure}
\begin{figure}
\centering
\includegraphics{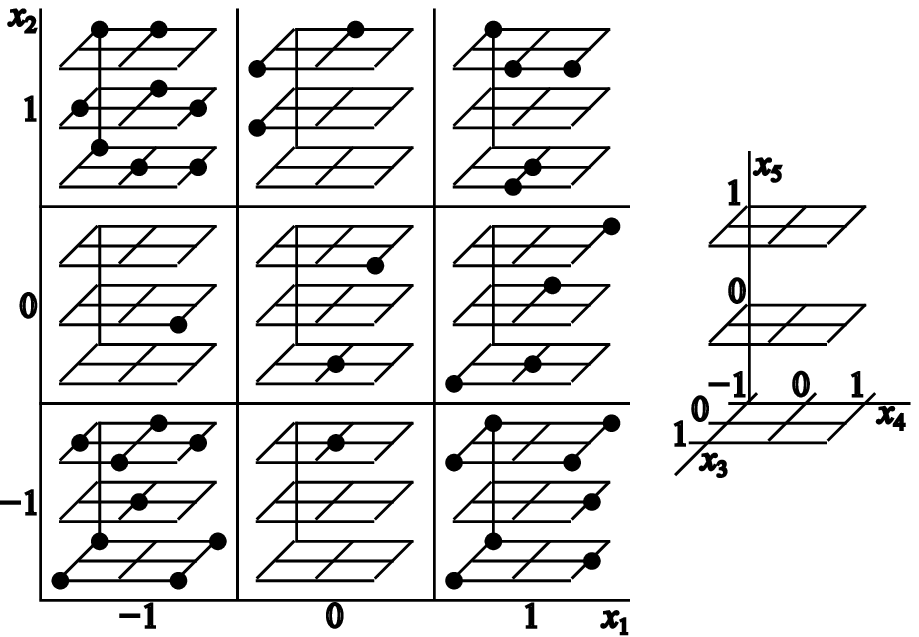}
\caption{The representative $41H$.\label{fign5s41H}}
\end{figure}
\begin{figure}
\centering
\includegraphics{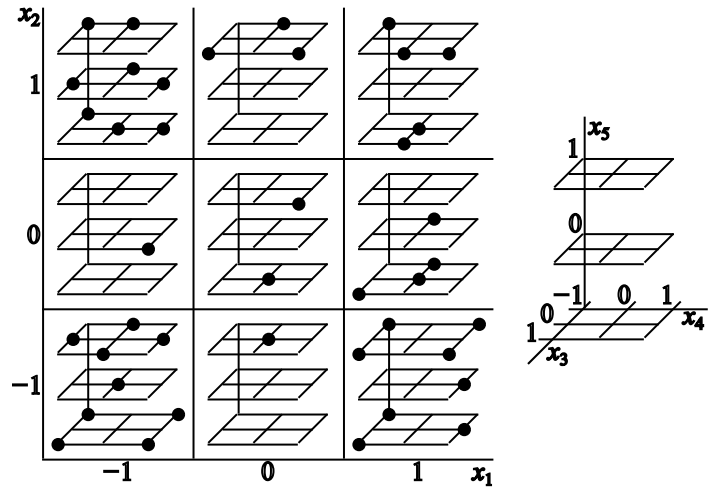}
\caption{The representative $41I$.\label{fign5s41I}}
\end{figure}
\begin{proposition} \label{n5s41201716}
Suppose that $C$ has no $\{20,20,1\}$ hyperplane directions and no $\{18,18,5\}$ hyperplane directions.
\begin{itemize}
\item[(a)] If some $4$-flat direction of $C$ has point count $\{20,17,4\}$ or $\{20,16,5\}$, then $C$ is a $\Delta{}686$ minus one cap point. If $C$ has a $4$-flat direction with point count $\{19,17,5\}$ or $\{19,16,6\}$, then $C$ is either a $\Delta{}686$ minus one cap point or a $41F$.
\item[(b)] If some $4$-flat direction of $C$ has point count $\{18,17,6\}$ or $\{18,16,7\}$, then $C$ is a $45$-cap $5$-flat minus four cap points, a $\Delta{}686$ minus one cap point, a $41F$, a $41G$, a $41H$, a $41I$, or a complete cap in which some saddled cube is the intersection of the $18$-cap $4$-flat in a $\{18,17,6\}$ or $\{18,16,7\}$ hyperplane direction and an $18$-, $17$-, or $16$-cap $4$-flat. (The lists in the supplementary files include all possible caps up to isomorphism for the last of those options; the lists may have more than one cap in each isomorphism class.)
\item[(c)] Suppose that the point count of some $4$-flat direction of $C$ is $\{17,17,7\}$ or $\{17,16,8\}$. It follows that $C$ is in case (a) or case (b).
\end{itemize}
\end{proposition}
\begin{proof}
(a) Suppose that for some co-ordinate $x_{1}$ of $C$, the numbers of cap points in the $4$-flats $x_{1} = -1$, $x_{1} = 1$, and $x_{1} = 0$ are $a$, $b$, and $c$ respectively where $20 \geq a \geq 19$ and $17 \geq b \geq 16$. By Lemma \ref{n4s17m3}, there is a co-ordinate $x_{2}$ of $C$ with $(x_{1},x_{2})$ independent such that the point count of $C$ for $(x_{1},x_{2})$ is
\[\left(\begin{array}{ccc}
9 & * & *\\
(1 \textrm{ or } 2) & * & *\\
9 & * & *
\end{array}\right).\]

As mentioned in Section \ref{secproofs41}, we obtain a list of possible $3$-flat point counts. We describe how one of the triples $(a,b,c)$ is dealt with; the other cases are handled in the same way. If $(a,b,c) = (19,16,6)$, then in the $4$-flat $x_{1} = 1$, the point count of the $x_{2}$-hyperplane direction is among $\{9,7,0\}$, $\{9,6,1\}$, $\{9,5,2\}$, $\{9,4,3\}$, $\{8,8,0\}$, \ldots, and $\{6,5,5\}$; each of those options gives a list of possible $3$-flat point counts of $C$ for $(x_{1},x_{2})$, after applying Lemma \ref{n4m3}. Some of those options are excluded by Lemmas \ref{n5clash}, \ref{n5865865}, and \ref{n5no919928}, and the fact that by Proposition \ref{n5s41}, the cap $C$ has no $\{x,y,z\}$ hyperplane direction such that $x \geq y \geq 18$. For example, if the $x_{2}$-hyperplane direction of $x_{1} = 1$ has point count $\{8,5,3\}$, then by Lemma \ref{n4m3}, without loss of generality the $3$-flat point count of $C$ is
\[\left(\begin{array}{ccc}
9 & d & 5\\
1 & e & 3\\
9 & f & 8
\end{array}\right)\]
with $d \leq 4$, $e \leq 2$, and $f \leq 2$; it follows that the $3$-flat point count of $C$ is among
\[\begin{array}{c}
\left(\begin{array}{ccc}
9 & 4 & 5\\
1 & 2 & 3\\
9 & 0 & 8
\end{array}\right),\left(\begin{array}{ccc}
9 & 4 & 5\\
1 & 1 & 3\\
9 & 1 & 8
\end{array}\right),
\left(\begin{array}{ccc}
9 & 4 & 5\\
1 & 0 & 3\\
9 & 2 & 8
\end{array}\right),\\
\left(\begin{array}{ccc}
9 & 3 & 5\\
1 & 2 & 3\\
9 & 1 & 8
\end{array}\right),\left(\begin{array}{ccc}
9 & 3 & 5\\
1 & 1 & 3\\
9 & 2 & 8
\end{array}\right)\textrm{, and }\left(\begin{array}{ccc}
9 & 2 & 5\\
1 & 2 & 3\\
9 & 2 & 8
\end{array}\right);
\end{array}\]
the second and third of those six matrices are excluded because they imply an $\{18,18,5\}$ hyperplane direction and a $\{19,18,4\}$ hyperplane direction respectively, and the last of the six matrices is excluded by Lemma \ref{n5no919928}.

By the computer search method of Section \ref{secproofs41}, it was shown that each of the $3$-flat point counts
\[\left(\begin{array}{ccc}
9 & 4 & 5\\
1 & 1 & 4\\
9 & 0 & 8
\end{array}\right)\textrm{ and }\left(\begin{array}{ccc}
9 & 4 & 4\\
1 & 1 & 4\\
9 & 1 & 8
\end{array}\right)\]
implies that $C$ is a $41F$, and (apart from $3$-flat point counts that are isomorphic to the two matrices above) each other $3$-flat point count in the list yields a contradiction or implies that $C$ is a $\Delta{}686$ minus a cap point.

(b) The proof of part (a) is adapted. Lemma \ref{n4s17m3} implies that for some pair $(x_{1},x_{2})$ of co-ordinates of $C$, the $4$-flat $x_{1} = 0$ has six or seven cap points and the point count of $C$ for $(x_{1},x_{2})$ is
\[\left(\begin{array}{ccc}
9 & * & *\\
0 & * & *\\
9 & * & *
\end{array}\right), \left(\begin{array}{ccc}
8 & * & *\\
1 & * & *\\
9 & * & *
\end{array}\right)\textrm{, or } \left(\begin{array}{ccc}
8 & * & *\\
2 & * & *\\
8 & * & *
\end{array}\right),\]
where in the last of those three cases, the bottom-left $8$ in the matrix corresponds to a saddled cube. Again, we obtain a list of possible $3$-flat point counts. If the hyperplane point count of $C$ for $x_{1}$ is $\{18,17,6\}$, then for each $3$-flat point count in the list, by applying the computer search method of Section \ref{secproofs41} we obtain (i) a contradiction; (ii) a hyperplane direction of $C$ with point count $\{20,20,1\}$, $\{18,18,5\}$, $\{20,17,4\}$, $\{20,16,5\}$, $\{19,17,5\}$, or $\{19,16,6\}$; or (iii) a cap $C$ among the options specified by the required result. If the hyperplane point count of $C$ for $x_{1}$ is $\{18,16,7\}$, then for each $3$-flat point count in the list, by applying the computer search method of Section \ref{secproofs41} we obtain (i) a contradiction; (ii) a hyperplane direction of $C$ with point count $\{20,20,1\}$, $\{18,18,5\}$, $\{20,17,4\}$, $\{20,16,5\}$, $\{19,17,5\}$, $\{19,16,6\}$, or $\{18,17,6\}$; or (iii) a cap $C$ among the options specified by the required result. (The supplementary files give the full results of the computer search.)

(c) This time, we obtain a $3$-flat point count of the form
\[\left(\begin{array}{ccc}
8 & * & *\\
0 & * & *\\
9 & * & *
\end{array}\right), \left(\begin{array}{ccc}
7 & * & *\\
1 & * & *\\
9 & * & *
\end{array}\right)\textrm{, or } \left(\begin{array}{ccc}
8 & * & *\\
1 & * & *\\
8 & * & *
\end{array}\right).\]
(Here, we make no assumptions about whether any $8$-cap $3$-flats are saddled cubes.) For each $3$-flat point count, by the computer search method of Section \ref{secproofs41}, we obtain a contradiction or there is a hyperplane direction of $C$ with point count $\{a,b,c\}$ where $a = b \in \{20,18\}$ or $20 \geq a \geq 18 \geq 17 \geq b \geq 16$.
\end{proof}
\begin{remark}
No two of the representatives $41A$ to $41I$ are isomorphic: see Table \ref{tbls41m4}. Each of those nine caps is complete and has no $\{20,20,1\}$ hyperplane directions.
\end{remark}
\begin{table}
\begin{tabular}{cccccccccc}
\hline
& \multicolumn{9}{c}{$41$-cap $5$-flat}\\
$4$-flat point count & $41A$ & $41B$ & $41C$ & $41D$ & $41E$ & $41F$ & $41G$ & $41H$ & $41I$\\
\hline
$\{18,18,5\}$ & $4$ & $2$ & $2$ & $1$ & $1$ & $0$ & $0$ & $0$ & $0$\\
$\{18,17,6\}$ & $0$ & $7$ & $4$ & $1$ & $1$ & $8$ & $5$ & $4$ & $2$\\
$\{18,16,7\}$ & $4$ & $2$ & $6$ & $8$ & $6$ & $3$ & $5$ & $5$ & $10$\\
\hline
\end{tabular}
\caption{Numbers of $\{18,18,5\}$, $\{18,17,6\}$, and $\{18,16,7\}$ hyperplane directions in some $41$-cap $5$-flats.\label{tbls41m4}}
\end{table}

The classification of all $5$-dimensional caps of size at least $41$ follows.

\begin{theorem}
Every $41$-cap $5$-flat is among the following $13$ options: a $45$-cap $5$-flat minus four cap points, a $\Delta{}686$ minus one cap point, a $41A$, a $41B$, \ldots, a $41I$, a complete cap with a $\{20,20,1\}$ hyperplane direction, or a complete cap in which there is no $\{18,18,5\}$ hyperplane direction and some saddled cube is the intersection of the $18$-cap $4$-flat in a $\{18,17,6\}$ or $\{18,16,7\}$ hyperplane direction and an $18$-, $17$-, or $16$-cap $4$-flat. Every $41$-cap $5$-flat that is not in the last of those $13$ options is in exactly one of the other $12$. Every $42$-cap $5$-flat is a $45$-cap $5$-flat minus three cap points or a $\Delta{}686$. Every $5$-dimensional cap of size at least $43$ can be obtained from a $45$-cap $5$-flat by removing at most two cap points.
\end{theorem}
\begin{proof}
Apply Proposition \ref{n5s41201716}, Lemma \ref{n5m4largesmall}, the remarks after those results, and Proposition \ref{n5s41}. The lists in the supplementary files include all possible caps up to isomorphism for the last of the $13$ options for $41$-cap $5$-flats; the lists may have more than one cap in each isomorphism class.
\end{proof}

\section*{Acknowledgements}

Many thanks to my postdoctoral supervisor Prof.\@ James Raftery, to Prof.\@ Roumen Anguelov, to Prof.\@ Jan Harm van der Walt, to Prof.\@ Mapundi Banda, to Prof.\@ Anton Str\"{o}h, and to everyone else at the University of Pretoria, for all their generous support at this extraordinary time.

This work was supported by the UP Post-Doctoral Fellowship Programme administered by the University of Pretoria [grant number A0X 816].

Computer searches were carried out using Java programs on the Eclipse IDE software.

Many thanks to my mother Dr Anne Thackeray for letting me use her computer together with my own to perform the search yielding Table \ref{tblcompsearch}, and to all of my family for everything.

\end{document}